\documentclass[12pt, reqno]{amsart}
 \usepackage{amsmath, amsthm, amscd, amsfonts, amssymb, graphicx, color, float}
\usepackage[bookmarksnumbered, colorlinks, plainpages]{hyperref}
\usepackage{tensor}
\input{mathrsfs.sty}
\hypersetup{colorlinks=true,linkcolor=red, anchorcolor=green, citecolor=cyan, urlcolor=red, filecolor=magenta, pdftoolbar=true}

\usepackage{graphicx}
\usepackage{tikz}
\usepackage{mathrsfs}
\usepackage{tkz-euclide}
\usetikzlibrary{decorations.text,calc,arrows.meta}
\usepackage{tkz-berge}
\newtheorem{theorem}{Theorem}[section]
\newtheorem{lemma}[theorem]{Lemma}
\newtheorem{proposition}[theorem]{Proposition}
\newtheorem{corollary}[theorem]{Corollary}
\theoremstyle{definition}
\newtheorem{definition}[theorem]{Definition}

\theoremstyle{remark}
\newtheorem{remark}[theorem]{Remark}
\numberwithin{equation}{section}

\begin{document}

\title [On approximately orthogonality preserving and reversing operators] {On approximately orthogonality preserving and reversing operators}
\author[D. Khurana ]{ Divya Khurana }

\address[Khurana]{IIM Ranchi, Prabandhan Nagar, Vill-Mudma, Nayasarai Road, Ranchi, Jharkhand-835303, India}
\email{divyakhurana11@gmail.com, divya.khurana@iimranchi.ac.in}

\renewcommand{\subjclassname}{\textup{2020} Mathematics Subject Classification}
\begin{abstract}
We study approximately orthogonality (in the sense of Dragomir) preserving and reversing operators. We show that for some orthogonality notations, an operator defined from a finite-dimensional Banach space to a normed linear space is approximately orthogonality preserving/reversing if and only if it is an injective operator. This result implies that for some orthogonality notations, any operator defined from an $n$-dimensional Banach space to another $n$-dimensional Banach space is approximately orthogonality preserving/reversing if and only if it is a scalar multiple of an $\varepsilon$-isometry.  We show that any $\varepsilon$-isometry and maps close to $\varepsilon$-isometries defined from a normed linear space to another normed linear space are approximately orthogonality preserving/reversing for some orthogonality notations. We also study the locally approximate orthogonality preserving and reversing operators defined on some finite-dimensional Banach spaces.
\end{abstract}
\maketitle
\section{Introduction}

 Throughout this article, symbols $X$ and $Y$ will be used for real normed linear spaces. We denote $B_X=\{x:x\in X, \|x\|\leq1\}$ and $S_X=\{x:x\in X, \|x\|=1\}$. We use the symbol $\mathcal{L}(X,Y)$  to denote the space of bounded linear operators from  $X$ to  $Y$. The dual space of $X$ will be denoted by $X^*$.

We denote $\mathcal{J}(x)=\{f: f\in S_{X^*}, f(x)=\|x\|\}$ for $0\not=x\in X$.  The Hahn-Banach theorem ensures that $\mathcal{J}(x)$ is always a non-empty set for each $0\not=x\in X$. If $\mathcal{J}(x)$ is a singleton set for some $0\not=x\in X$, then $x$ is said to be a smooth point.  

 For $x,y\in X$, $x$ is said to be  Birkhoff-James orthogonal to $y$ (denoted by $x\perp_By$) if $\|x+\lambda y\|\geq \|x\|$ for all $\lambda\in\mathbb{R}$ (see \cite{Birkhoff}, \cite{J} and \cite{James2}).

 There are various other notions of orthogonality and approximate Birkhoff-James orthogonality. We recall a few of these notions related to our work. Let $X$ be a normed linear space and $x,y\in X$. If $\|x+y\|= \|x-y\|$ then we say that $x$ is isosceles orthogonal to $y$ (denoted by $x\perp_I y$)(see \cite{J}); if  $\|x+\lambda y\|= \|x-\lambda y\|$ for all $\lambda\in\mathbb{R}$ then we say that $x$ is Roberts orthogonal to $y$ (denoted by $x\perp_R y$) (see \cite{Robert}). 
 
 Dragomir introduced the following notion of approximate  Birkhoff-James orthogonality in \cite{Drag}.  Let $X$ be a normed linear space, $x,y\in X$ and $\varepsilon\in[0,1)$, we say that $x$ is approximately Birkhoff-James orthogonal to $y$ if $\|x+\lambda y\|\geq (1-\varepsilon) \|x\|$ for all scalars $\lambda\in\mathbb{R}$. The following modification in this definition was given by Chmieli\'{n}ski (see \cite{C}).  
Let $X$ be a normed linear space, $x,y\in X$ and $\varepsilon\in[0,1)$, then $x$ is said to be
approximately Birkhoff-James orthogonal to $y$ (denoted by $x\perp_{D^\varepsilon} y$) if $\|x+\lambda y\|\geq
\sqrt{1-\varepsilon^2}\|x\|$ for all $\lambda\in\mathbb{R}$. 

The following definition of approximate Birkhoff-James orthogonality was also introduced by Chmieli\'{n}ski in \cite{C}. Let $X$ be a normed linear space, $x,y \in X$ and $\varepsilon\in
[0,1)$, we say that  $x$ is approximately orthogonal to $y$ (denoted by $x\perp_{B^\varepsilon} y$) if $\|x+\lambda y\|^2\geq
\|x\|^2-2\varepsilon\|x\|\|\lambda y\|$ for all $\lambda \in
\mathbb{R}$.

Let $X$ be a normed linear space, $\varepsilon\in[0,1)$ and  $\Delta\in\{R, B,B^\varepsilon, D^\varepsilon\}$. Then $\perp_\Delta$ satisfies the homogeneity property, i.e., $x\perp_\Delta y$ if and only if $\alpha x\perp_\Delta \beta y$ for all $x,y\in X$ and $\alpha,\beta\in\mathbb{R}$. Thus, while working with any $\Delta\in\{R, B,B^\varepsilon, D^\varepsilon\}$ we can restrict ourself to $S_X$ for $\varepsilon\in[0,1)$. By using the notations of \cite{WHY}, if we restrict to $x,y\in S_X$ in the definition of the isosceles orthogonality, then we will denote it as $x\perp_{UI} y$.

For $x\in S_X$, $\varepsilon\in[0,1)$ and  $\Delta\in\{UI, R, B,B^\varepsilon, D^\varepsilon\}$, we denote $x^{\perp_\Delta}=\{y\in S_X: x\perp_\Delta y\}$.

The following are equivalent characterizations of some of the orthogonality notations used in our work. Let $X$ be a normed linear space and $0\not=x\in X$, $y\in X$. Then 

\begin{itemize}
    
\item  $x\perp_B y$ $\iff$ $f(y)=0$ for some $f\in \mathcal{J}(x)$ (\cite[Theorem 2.1]{James2});

\item   for $\varepsilon\in[0,1)$, $x\perp_{D^\varepsilon} y$ $\iff$ there exists $f\in S_{X^*}$ such that $\ |f(x)|\geq \sqrt{1-\varepsilon^2}\|x\|\mbox{~ and~} f(y)=0$ (\cite[Lemma 3.2]{MSP});

\item for $\varepsilon\in[0,1)$, $x\perp_{B^\varepsilon} y \iff |f(y)|\leq \varepsilon \|y\| \mbox{~for~some~}f\in \mathcal{J}(x)$ (\cite[Theorem 2.3] {CSW}).

\end{itemize}

Let $X,Y$ be normed linear spaces and $T\in \mathcal{L}(X,Y)$. Then $T$ is said to be orthogonality preserving (reversing) if $x\perp_B y$ implies $Tx\perp_B Ty$ ($Ty\perp_B Tx)$ for all $x,y\in X$.  If $T$ is an orthogonality preserving operator defined on a Banach space $X$ then it was proved in  \cite{T4} that  $T$ is a scalar multiple of a linear isometry.  This result is also true if $T$ is a bounded linear operator defined from a complex normed linear space $X$ to a complex normed linear space $Y$ (see \cite{T3} for details). In \cite{WOR} and \cite{WHY} it was proved that if $X$ is a normed linear space with $dim~X\geq 3$ then $X$ has a non-zero orthogonality reversing operator if and only if $X$ is an inner product space. For more details on orthogonality preserving and reversing operators, we refer the readers to  \cite{T3}, \cite{CAOT}, \cite{GPSBOOKC}, \cite{T4}, \cite{WOR}, \cite{WHY} and the references cited therein.

The following version of approximate orthogonality preserving and reversing operators was introduced in \cite{AOT}.

\begin{definition}
    Let $X$ and $Y$ be normed linear spaces. Let $T\in \mathcal{L}(X,Y)$ and $\varepsilon\in[0,1)$.  We say that $T$ is $\varepsilon$-orthogonality preserving (reversing), if $x\perp_B y$ implies $Tx\perp_{D^\varepsilon} Ty$ ($Ty\perp_{D^\varepsilon }Tx$) for all $x,y\in X$.
\end{definition}

Recently, local orthogonality preserving operators defined on finite-dimensional Banach spaces have been studied in \cite{SMP}.

Motivated by these definitions of orthogonality preserving/reversing operators we now introduce the following notions of approximately orthogonality preserving/reversing operators and locally approximately orthogonality preserving/reversing operators. 

\begin{definition}\label{def1}
    Let $X,Y$ be normed linear spaces and $T\in \mathcal{L}(X,Y)$. Let $\varepsilon,\eta\in[0,1)$ and  $\Delta\in\{I, R, B,B^\varepsilon, D^\varepsilon\}$. We say that $T$ is $\Delta D^\eta$-orthogonality preserving (reversing), in short, $\Delta D^\eta$-P ($\Delta D^\eta$-R), if $x\perp_\Delta y$ implies $Tx\perp_{D^\eta} Ty$ ($Ty\perp_{D^\eta }Tx$) for all $x,y\in X$.
\end{definition}

\begin{definition}\label{def2}
    Let $X,Y$ be normed linear spaces and $T\in \mathcal{L}(X,Y)$. Let $x\in X$, $\varepsilon,\eta\in[0,1)$ and  $\Delta\in\{I, R, B,B^\varepsilon, D^\varepsilon\}$. We say that $T$ is $\Delta D^\eta$-orthogonality preserving (reversing) at $x$, in short, $\Delta D^\eta$-P ($\Delta D^\eta$-R) at $x$, if $x\perp_\Delta y$ implies $Tx\perp_{D^\eta} Ty$ ($Ty\perp_{D^\eta }Tx$) for all $y\in X$.
\end{definition}

If we take $\Delta=UI$, then to define the analogous version of Definitions~\ref{def1} and \ref{def2} we need to restrict ourselves to $x,y\in S_X$ in the Definitions~\ref{def1} and \ref{def2}.

The following complete characterization of $BD^\eta$-P and $BD^\beta$-R operators defined from a finite-dimensional Banach space to a normed linear space has been observed in \cite{AOT}, where $\eta,\beta\in[0,1)$.

\begin{theorem}\label{AOT}\cite[Theorem 3.4]{AOT}
    Let $X$ be a finite-dimensional Banach space. Let $Y$ be a normed linear space and $T\in \mathcal{L}(X,Y)$. Then the following are equivalent.

    \begin{itemize}
        \item[(i)] $T$ is $B D^\eta$-P for some $\eta\in[0,1)$.
        \item [(ii)] $T$ is $B D^\beta$-R for some $\beta\in[0,1)$.
        \item[(iii)] $T$ is one-to-one.
    \end{itemize}
\end{theorem}

In this paper, we study $\Delta D^\eta$-P and $\Delta D^\eta$-R operators, where $\Delta\in\{I,UI,R,B,B^\varepsilon,\\D^\varepsilon\}$ and $\eta,\varepsilon\in[0,1)$. We show that Theorem~\ref{AOT} is also true if we take $\Delta\in\{UI, B^\varepsilon,D^\varepsilon\}$, where $\varepsilon\in[0,1)$. This result implies that any operator defined from an $n$-dimensional Banach space to another $n$-dimensional Banach space is $\Delta D^\eta$-P if and only if it is a scalar multiple of an  $\varepsilon$-isometry, where $\Delta\in\{UI,B^\beta,D^\beta\}$ and $\beta,\eta\in[0,1)$. We also study operators defined on some finite-dimensional Banach spaces for which locally $B D^\varepsilon$-P at some elements implies that the operator is $B D^\eta$-P, where $\varepsilon,\eta\in[0,1)$.

\section{$\varepsilon$-isometries and approximately orthogonality preserving/reversing operators}
The following result was proved in \cite{CKS2}.

\begin{proposition}\cite[Proposition 2.2]{CKS2}
    Let $X$ be a normed linear space. Let $x,y\in S_X$ be linearly independent elements. Then there exists $\varepsilon_{x,y}\in[0,1)$ such  $x\perp_{D^{\varepsilon_{x,y}}}y$.
\end{proposition}

We now give another proof of this result that provides a better estimate of $\varepsilon_{x,y}\in[0,1)$ in \cite[Proposition 2.2]{CKS2}.
\begin{proposition}\label{independent}
    Let $X$ be a normed linear space and $x,y\in S_X$ be linearly independent elements.  Let $\alpha=\inf_{\lambda\in \mathbb{R}} \|x+\lambda y\|$. Then $0<\alpha\leq 1$ and  $x\perp_{D^{\varepsilon_{x,y}}}y$, where $\sqrt{1-\alpha^2}=\varepsilon_{x,y}$.
\end{proposition}
\begin{proof}
    Let $\alpha=\inf_{\lambda\in \mathbb{R}} \|x+\lambda y\|$. We claim $0<\alpha\leq 1$. 

    Observe that $\alpha\leq\|x\|=1$. If $\alpha=0$ then there exists $\{\lambda_n\}\subseteq\mathbb{R}$ such that $\lim_{n\longrightarrow\infty}\|x+\lambda_n y\|=0$. Now, $\lim_{n\longrightarrow\infty}\|x+\lambda_n y\|=0$ implies that $\{\lambda_n\}$ can not be an unbounded set. Thus, $\{\lambda_n\}$ is a bounded set and without loss of generality, we assume that $\lambda_n\longrightarrow\lambda$ for some $\lambda\in\mathbb{R}$. This gives $\lim_{n\longrightarrow\infty}\|x+\lambda_n y\|=\|x+\lambda y\|=0$. Thus, $x+\lambda y=0$ and $x,y\in S_X$ implies that either $x=y$ or $x=-y$, a contradiction. Thus, $\alpha\not=0$, and this proves our claim. Let $\varepsilon_{x,y}\in[0,1)$ such that $\sqrt{1-\alpha^2}=\varepsilon_{x,y}$. Then $\|x+\lambda y\|\geq \sqrt{1-\varepsilon^2_{x,y}}\|x\|$ for all $\lambda\in \mathbb{R}$ and thus $x\perp_{D^{\varepsilon_{x,y}}}y$.
    \end{proof}

Let $X,Y$ be normed linear spaces and $T\in \mathcal{L}(X,Y)$ such that $T$ is $\Delta D^{\eta}$-P for some $\eta\in[0,1)$, where $\Delta\in \{I,R\}$. Then by using symmetry of $\Delta$ orthogonality, we get $x\perp_\Delta y$ implies $y\perp_\Delta x$ for all $x,y\in X$. Thus, if $x\perp_\Delta y$ then $Tx\perp_{D^\eta} Ty$  and $Ty\perp_{D^\eta} Tx$. We now prove a similar result for $\Delta\in \{B, B^\varepsilon, D^\varepsilon\}$, where $\varepsilon\in[0,1)$.
\begin{proposition}\label{prop1}
    Let $\varepsilon\in[0,1)$ and  $\Delta\in\{B,B^\varepsilon, D^\varepsilon\}$. Let $X$ be a finite-dimensional Banach space and $Y$ be a normed linear space. If $T\in S_{\mathcal{L}(X,Y)}$ then the following are equivalent.
    \begin{itemize}
        \item[(i)] $T$ is $\Delta D^{\eta}$-P for some $\eta\in[0,1)$.
        \item[(ii)] $T$ is $\Delta D^{\beta}$-R for some $\beta\in[0,1)$.
    \end{itemize}
\end{proposition}

\begin{proof}
   (i)$\implies$ (ii) Let $x\in S_X$ and $y\in x^{\perp_\Delta}$. If $Tx=0$ or $Ty=0$ then $Ty\perp_{D^{\gamma}} Tx$ for any $\gamma\in[0,1)$. Let  $Tx, Ty\not=0$. Then (i) implies $Tx\perp_{D^{\eta}} Ty$. Thus, $Tx$ and $Ty$ are linearly independent. Now, the homogeneity property of $D^\eta$ and Proposition~\ref{independent} implies that there exists $\varepsilon_{Ty,Tx}\in[0,1)$ such that  $Ty\perp_{D^{\varepsilon_{Ty,Tx}}} Tx$. For $x\in S_X$ and $y\in x^{\perp_\Delta}$ let $\varepsilon_{Ty,Tx}^*$ be the infimum of all $\varepsilon_{Ty,Tx}$ such that $Ty\perp_{D^{\varepsilon_{Ty,Tx}}} Tx$. We claim that $\sup_{x\in S_X}\sup_{y\in  x^{\perp_\Delta}}\varepsilon_{Ty,Tx}^*<1$. If $\sup_{x\in S_X}\sup_{y\in  x^{\perp_\Delta}}\varepsilon_{Ty,Tx}^*=1$ then there exist $\{x_n\},\{y_n\}\subseteq S_X$ and an increasing sequence $\varepsilon_n\nearrow 1$ such that $x_n\perp_\Delta y_n$, $Ty_n\perp_{D^{\varepsilon_n}}Tx_n$ and $Ty_n\not\perp_{D^{\eta_n}}Tx_n$ for any $\eta_n<\varepsilon_n$. 
   
   Now, using the compactness of $S_X$ we can obtain convergent subsequences of $\{x_n\}$ and  $\{y_n\}$. Without loss of generality, we assume that $x_{n}\longrightarrow x$ and $y_{n}\longrightarrow y$ for some $x,y\in S_X$. Thus, we get $Tx_n\longrightarrow Tx$ and $Ty_n\longrightarrow Ty$. (i) implies that $Tx_n\perp_{D^{\eta}} Ty_n$  for all $n\in \mathbb{N}$. Now, using $Tx_n\longrightarrow Tx$, $Ty_n\longrightarrow Ty$ and continuity of the norm, we get $Tx\perp_{D^{\eta}} Ty$. We claim that $Tx\not=0$ and $Ty\not=0$. If $Tx=0$ then by using $T\in S_{\mathcal{L}(X,Y)}$, we get $Tz\not=0$ for some $z\in S_X$. Clearly, $x$ and $z$ are linearly independent. Let $M=\mbox{span}\{x,z\}$. Then \cite{ratz} implies that there exists $0\not=u\in M$ such that $x\perp_Bu$ and $x+u\perp_B x-u$. Thus, $x\perp_\Delta u$ and $x+u\perp_\Delta x-u$ for $\Delta\in \{B, B^\varepsilon, D^\varepsilon\}$, where $\varepsilon\in[0,1)$. If $u=b z$ for some $b\not=0$ then $x+u\perp_\Delta x-u$ and (i) imply $b Tz\perp_{D^\eta} -b Tz$, a contradiction. Thus, $u=a x+b z$ for some scalars $a,b\not=0$. Again, $x+u\perp_\Delta x-u$ and (i) imply $b Tz\perp_{D^\eta} -b Tz$, a contradiction. This shows that $Tx\not=0$. Similarly, we can show that $Ty\not=0$.
   
   Thus, (i) implies that $Tx$, $Ty$ are linearly independent. By the homogeneity property of $D^\eta$ and Proposition~\ref{independent} there exists $\delta_1\in[0,1)$ such that $Ty\perp_{D^{\delta_1}}Tx$. Now, using the homogeneity property of $D^\eta$  and \cite[Theorem 2.1]{CKS2} we can find $k\in\mathbb{N}$ and $\delta_2\in[0,1)$ such that $Ty_m\perp_{D^{\delta_2} }Tx_m$ for all $m\geq k$, a contradiction. Thus, (ii) follows.

   (ii)$\implies$ (i) follows by using the similar arguments used to prove (i)$\implies$ (ii).
\end{proof}

It follows from Proposition~\ref{prop1} that for a finite-dimensional Banach space, the Dragomir orthogonality is symmetric in the following sense.

\begin{corollary}
    Let $X$ be a finite-dimensional Banach space and $\varepsilon\in[0,1)$. Then there exists $\beta\in[0,1)$ such that whenever $x,y\in S_X$ and  $x\perp_{D^\varepsilon} y$ implies $y\perp_{D^\beta}x$.
\end{corollary}

We now prove that any injective operator defined from a finite-dimensional Banach space to a normed linear space is $\Delta D^{\eta}$-P and $\Delta D^{\beta}$-R for some $\eta,\beta\in[0,1)$, where $\Delta\in\{UI,R, B,B^\varepsilon, D^\varepsilon\}$ and $\varepsilon\in[0,1)$.

\begin{proposition}\label{PR}
    Let $\varepsilon\in[0,1)$ and  $\Delta\in\{UI,R, B,B^\varepsilon, D^\varepsilon\}$. Let $X$ be a finite-dimensional Banach space and $Y$ be normed linear space. If $T\in S_{\mathcal{L}(X,Y)}$ is one-to-one, then the following results hold.
    \begin{itemize}
        \item[(i)] $T$ is $\Delta D^{\eta}$-P for some $\eta\in[0,1)$.
        \item[(ii)] $T$ is $\Delta D^{\beta}$-R for some $\beta\in[0,1)$.
    \end{itemize}
\end{proposition}
\begin{proof}
     (i) Let $x\in S_X$ and $y\in x^{\perp_\Delta}$. Then $x,y$ are linearly independent. Thus, $Tx$ and $Ty$ are linearly independent. Using the homogeneity property of $D^\eta$ and Proposition~\ref{independent} there exists $\varepsilon_{Ty,Tx}\in[0,1)$ such that $Ty\perp_{D^{\varepsilon_{Ty,Tx}}} Tx$. For $x\in S_X$ and $y\in x^{\perp_\Delta}$ let $\varepsilon_{Ty,Tx}^*$ be the infimum of all $\varepsilon_{Ty,Tx}$ such that $Ty\perp_{D^{\varepsilon_{Ty,Tx}}} Tx$. Now, by using the arguments similar to Proposition~\ref{prop1}, we can show that $\sup_{x\in S_X}\sup_{y\in x^{\perp_\Delta}}\varepsilon_{Ty,Tx}^*<1$. Thus, (i) follows.

    (ii) follows using the similar arguments used to prove (i).
\end{proof}

Let $X,Y$ be normed linear spaces. Then an operator $T\in S_{\mathcal{L}(X,Y)}$ is said to be bounded from below if there exists $\delta>0$ such that $\delta \|x\|\leq \|Tx\|$ for all $x\in X$. 
We now show that any $\Delta D^\eta$-P linear operator defined from a normed linear space to another normed linear space is always a continuous and bounded from below operator, where $\Delta\in\{B,B^\varepsilon, D^\varepsilon\}$ and $\eta,\varepsilon\in [0,1)$. To prove this result, we will use some ideas from \cite[Lemma 3.2]{MT}.
        \begin{lemma}\label{bounded below}
    Let $\varepsilon,\eta\in[0,1)$ and  $\Delta\in\{B,B^\varepsilon, D^\varepsilon\}$. Let $X, Y$ be normed linear spaces and $T: X\longrightarrow Y$ be a non-zero linear operator.  If $T$ is $\Delta D^{\eta}$-P, then $T$ is continuous and bounded from below.
     \end{lemma}

     \begin{proof}
         Let $x,y\in S_X$. Then $3x+y\not\perp_By$. Let $f\in \mathcal{J}(3x+y)$ and $\alpha=\dfrac{f(3x)}{\|3x+y\|}$. Then $3x+y\perp_B3x-\alpha(3x+y)$. Observe that $\|3x+y\|\leq 4$, $\|3x+y\|\geq |\|3x\|-\|y\||=2$, $|f(3x)|=|f(3x+y-y)|=|\|3x+y\|-f(y)|\geq 1$ and $|f(3x)|\leq 3$. Now, by using  $2\leq\|3x+y\|\leq 4$ and $1\leq |f(3x)|\leq 3$, we get $\frac{1}{4}\leq |\alpha| \leq \frac{3}{2}$. 
         
         Since $3x+y\perp_B3x-\alpha(3x+y)$, we get   $3x+y\perp_\Delta 3x-\alpha(3x+y)$ for $\Delta\in\{B,B^\varepsilon, D^\varepsilon\}$. Now, by using $T$ is $\Delta D^\eta$-P for $\Delta\in\{B,B^\varepsilon, D^\varepsilon\}$, we get
         \begin{align*}
      \|T(3x+y)+\lambda T(3x-\alpha(3x+y))\|&\geq \sqrt{1-\eta^2}\|T(3x+y)\|. 
    \end{align*}
        By choosing $\lambda=\frac{1}{\alpha}$, we get
        \begin{align*}
      \frac{3}{|\alpha|}\|Tx\|&\geq \sqrt{1-\eta^2}\|T(3x+y)\|. 
    \end{align*}
       Thus, 
        \begin{align*}
        12\|Tx\|\geq \frac{3}{|\alpha|}\|Tx\|&\geq \sqrt{1-\eta^2}\|T(3x+ y)\|\geq \sqrt{1-\eta^2}(\|Ty\|-3\|Tx\|).
         \end{align*}
        This shows 
         \begin{align}\label{eq1}
       \|Tx\|\geq \frac{\sqrt{1-\eta^2}}{12+3\sqrt{1-\eta^2}}\|Ty\|.
         \end{align}

         Since $x,y\in S_X$ are arbitrary, Eqn.~(\ref{eq1}) implies $T$ is bounded.

         By taking supremum over all $y\in S_X$ on the right-hand side of Eqn.~(\ref{eq1}), we get 
         \begin{align*}
       \|Tz\|\geq \frac{\sqrt{1-\eta^2}}{12+3\sqrt{1-\eta^2}}\|T\|\|z\|
         \end{align*}
         for all $z\in X$.
         This proves the result.
     \end{proof}

It follows from Lemma~\ref{bounded below} that if $X,Y$ are normed linear spaces and $0\not=T\in\mathcal{L}(X,Y)$ is $\Delta D^\eta$-P operator then $T$ is an injective operator, where $\Delta\in\{B,B^\varepsilon, D^\varepsilon\}$ and $\varepsilon,\eta\in[0,1)$. We now prove a similar result for $\Delta\in\{I, UI\}$. 

   \begin{lemma}\label{r3}
       Let $\eta\in[0,1)$ and $\Delta\in\{I,UI\}$. Let $X$ and $Y$ be normed linear spaces. Let $T\in S_{\mathcal{L}(X,Y)}$ be such that $T$ is $\Delta D^\eta$-P. Then $T$ is one-to-one. 
   \end{lemma}

   \begin{proof}
   Since $T\in S_{\mathcal{L}(X,Y)}$, we can find an $x\in S_X$ such that $Tx\not=0$. If $T$ is not one-to-one, then there exists $y\in S_X$ such that $Ty=0$. Clearly, $x$ and $y$ are linearly independent. 

    We first consider $\Delta=I$. Clearly, $\dfrac{x+y}{2}\perp_\Delta \dfrac{x-y}{2}$ and this gives $T(\frac{x+y}{2})\perp_{D^\eta} T (\frac{x-y}{2})$. Thus, $Tx\perp_{D^\eta} Tx$, a contradiction. This shows that $T$ is one-to-one.

   Now, we consider $\Delta=UI$. Let $M=span\{x,y\}$ and $u\in S_{M}$ such that $\|u-y\|=\frac{1}{2}$. Then it follows from \cite[Lemma 2.2]{GL} that $u\not\perp_\Delta y$. Also, $u$ and $y$ are linearly independent. Let $u=\alpha_1 x+\beta_1 y$, where $\alpha_1,\beta_1\in\mathbb{R}$ and $\alpha_1\not=0$.

   By using \cite[Lemma 2.2]{GL} we can find a $v\in S_{M}$ such the $u\perp_\Delta v$. Clearly, it follows that $v$ and $y$ are linearly independent. Let $v=\alpha_2 x+\beta_2 y$, where $\alpha_2,\beta_2\in\mathbb{R}$ and $\alpha_2\not=0$. Now, by given assumption on $T$, we get $Tu\perp_{D^\eta} Tv$. Since $Tu,Tv\not=0$, we get $Tu$ and $Tv$ are linearly independent. But $Ty=0$ implies that $Tu=\alpha_1 Tx$ and $Tv=\alpha_2 Tx$, a contradiction. Thus, $T$ is one-to-one. 
   \end{proof}

   \begin{remark}
   Let $x=\left(\dfrac{3}{4},-\dfrac{1}{4}\right)\in \ell_1^2$ and $y=(y_1,y_2)\in S_{\ell_1^2}$ such that $x\perp_R y$. Since $x\perp_R y$ implies $x\perp_B y$, so we get $y_1-y_2=0$. Thus, any such $y\in S_{\ell_1^2}$ is either $y=\left(\dfrac{1}{2},\dfrac{1}{2}\right)$ or  $y=\left(-\dfrac{1}{2},-\dfrac{1}{2}\right)$. But for any of these values of $y$, $\|x+y\|\not=\|x-y\|$. Thus, for $x=\left(\dfrac{3}{4},-\dfrac{1}{4}\right)\in \ell_1^2$ there does not exists any $0\not=y\in\ell_1^2$ such that $x\perp_R y$. This shows that in general, Roberts orthogonality fails to have the existence property. The given proof of Lemma~\ref{r3} uses the existence property of the considered orthogonality relations, thus arguments of Lemma~\ref{r3} can not be used for $\Delta=R$.
   \end{remark}

It follows from Theorem~\ref{AOT} that $T\in S_{\mathcal{L}(X,Y)}$ is $BD^\eta$-P/$BD^\eta$-R if and only if $T$ is injective, where $X$ is a finite-dimensional Banach space, $Y$ is a normed linear space and $\eta\in[0,1)$. In the following result we show a similar result for any $\Delta D^\eta$-P operator, where $\Delta\in\{UI,B,B^\varepsilon, D^\varepsilon\}$ and $\eta,\varepsilon\in[0,1)$. This result extends Theorem~\ref{AOT}.

\begin{theorem}\label{one}
Let $\varepsilon\in[0,1)$ and $\Delta\in\{UI,B,B^\varepsilon, D^\varepsilon\}$. Let $X$ be a finite-dimensional Banach space and $Y$ be a normed linear space. Let $T\in S_{\mathcal{L}(X,Y)}$. Then the following are equivalent.
    \begin{itemize}
        \item[(i)] $T$ is one-to-one.
        \item [(ii)]$T$ is $\Delta D^{\beta}$-R for some $\beta\in[0,1)$.
        \item[(iii)] $T$ is $\Delta D^{\eta}$-P for some $\eta\in[0,1)$.
    \end{itemize}
     \end{theorem}
\begin{proof}
(i) implies (ii) follows from the Proposition~\ref{PR}. For $\Delta\in\{B,B^\varepsilon, D^\varepsilon\}$ (ii) implies (iii) follows from the Proposition~\ref{prop1} and for $\Delta=UI$ (ii) implies (iii) follows from the symmetry of $\Delta$.
Lemmas~\ref{bounded below} and \ref{r3} show that (iii) implies (i).
\end{proof}

We now recall the definition of an $\varepsilon$-isometry. Let $X,Y$ be normed linear spaces. A linear map $T:X\longrightarrow Y$ is said to be an $\varepsilon$-isometry if $(1-\delta_1(\varepsilon))\|x\|\leq\|Tx\|\leq  (1+\delta_2(\varepsilon))\|x\|$, where $\lim\limits_{\varepsilon\longrightarrow 0}\delta_1(\varepsilon)=0$ and $\lim\limits_{\varepsilon\longrightarrow 0}\delta_2(\varepsilon)= 0$. Let $\beta,\eta\in[0,1)$ and $\Delta\in\{UI,B,B^\beta, D^\beta\}$. Let $X$, $Y$ be $n$-dimensional Banach spaces. The following result shows that an operator $T\in S_{\mathcal{L}(X,Y)}$ is $\Delta D^{\eta}$-P/$\Delta D^\eta$-R if and only if $T$ is an $\varepsilon$-isometry. 

\begin{corollary}\label{almost isometry}
   Let $\beta\in[0,1)$ and $\Delta\in\{UI,B,B^\beta, D^\beta\}$.  Let $X$ and $Y$ be $n$-dimensional Banach spaces. Let $T\in S_{\mathcal{L}(X,Y)}.$ Then the following are equivalent:
   \begin{itemize}
       \item [(i)] $T$ is an $\varepsilon$-isometry.
       \item [(ii)]$T$ is $\Delta D^{\delta}$-R for some $\delta\in[0,1)$. 
       \item [(iii)] $T$ is $\Delta D^{\eta}$-P for some $\eta\in[0,1)$.
    \end{itemize}
\end{corollary}
\begin{proof}
    (i) $\implies$ (ii) and (ii) $\implies$ (iii) follows from Theorem~\ref{one}. 
    
    (iii)$\implies$ (i) It follows from Theorem~\ref{one} that $T$ is one-to-one operator. Thus, by using $X$ and $Y$ are $n$-dimensional Banach spaces, we get $T$ is an invertible operator. Now, $\|T\|=1$, implies $\|T^{-1}\|\geq 1$. Also,
    $${\|T^{-1}\|}^{-1}\|x\|\leq \|Tx\|\leq \|x\|$$
    for all $x\in X$.
    Let  ${\|T^{-1}\|}^{-1}=1-\varepsilon$ for some  $\varepsilon\in[0,1)$. Then,  $$(1-\varepsilon)\|x\|\leq \|Tx\|\leq \|x\|$$ for all $x\in X$ and thus (i) follows.

\end{proof}

We now recall the following definition and result that helps us to reformulate Corollary~\ref{almost isometry} in terms of maps which are close to isometries. Let $X,Y$ be normed linear spaces. A linear map $T:X\longrightarrow Y $ is said to be close to an isometry if $\|T-I\|\leq \delta(\varepsilon)$, where $\lim\limits_{\varepsilon\longrightarrow 0}\delta(\varepsilon)= 0$ and $I:X\longrightarrow Y$ is an isometry. Pair $(X,Y)$  is said to have the stability of linear isometries property if there exists a function $\delta:[0,1)\longrightarrow \mathbb{R}_+$ with $\delta(\varepsilon)\longrightarrow 0$ as $\varepsilon\longrightarrow 0^+$ such that for any $\varepsilon$-isometry $T:X\longrightarrow Y$ there exists an isometry $I:X\longrightarrow Y$ such that $\|T-I\|\leq \delta(\varepsilon)$.

By using the result that if $X$ and $Y$ are finite-dimensional Banach spaces then the pair $(X,Y)$ has the stability of linear isometries property (see \cite{Ding} for details) and linear maps close to isometries are $\varepsilon$-isometries, we can reformulate Corollary~\ref{almost isometry} in the following way.
\begin{corollary}
   Let $\beta\in[0,1)$ and $\Delta\in\{UI,B,B^\beta, D^\beta\}$.  Let $X$ and $Y$ be $n$-dimensional Banach spaces.  Let $T\in S_{\mathcal{L}(X,Y)}.$ Then the following are equivalent:
   \begin{itemize}
       \item [(i)] $T$ is close to an isometry.
       \item [(ii)]$T$ is $\Delta D^{\delta}$-R for some $\delta\in[0,1)$. 
       \item [(iii)] $T$ is $\Delta D^{\eta}$-P for some $\eta\in[0,1)$.
    \end{itemize}
\end{corollary}

We now study a result similar to Corollary~\ref{almost isometry} by removing the assumptions that $X$ and $Y$ are $n$-dimensional Banach spaces. Let $\beta\in[0,1)$ and $X,Y$ be normed linear spaces. We now show that scalar multiple of $\varepsilon$-isometries defined from $X$ to $Y$ are always $D^\beta D^\eta$-P, $B^\beta D^\eta$-P, $R D^\eta$-P and $R D^\eta$-R for some $\eta\in[0,1)$.
\begin{theorem}\label{isometry}
    Let $X,Y$ be normed linear spaces and $\beta\in[0,1)$. Let $0
\not=T\in \mathcal{L}(X,Y)$ be a scalar multiple of an $\varepsilon$-isometry. Then
    \begin{itemize}
        \item[(i)] $T$ is $D^\beta D^\eta$-P, where $\eta=\sqrt{1-\left(\dfrac{1-\delta_1(\varepsilon)}{1+\delta_2(\varepsilon)}\right)^2(1-\beta^2)}$.
        \item[(ii)] $T$ is $B^\beta D^\eta$-P for $\eta=\sqrt{1-\left(\dfrac{1-\delta_1(\varepsilon)}{1+\delta_2(\varepsilon)}\right)^2{\left(\dfrac{1-\beta}{1+\beta}\right)^2}}$.
        \item[(iii)] $T$ is $R D^\eta$-P and $R D^\eta$-R for $\eta=\sqrt{1-\left(\dfrac{1-\delta_1(\varepsilon)}{1+\delta_2(\varepsilon)}\right)^2}$.
    \end{itemize}
\end{theorem}

\begin{proof}
Let $T=\alpha U$, where $U$ is an $\varepsilon$-isometry and $\alpha\not=0$. Then 
    $$|\alpha|(1-\delta_1(\varepsilon))\|x\|\leq \|Tx\|\leq |\alpha|(1+\delta_2(\varepsilon))\|x\|$$
    for all $x\in X$, where $\delta_1(\varepsilon)\longrightarrow 0$ and $\delta_2(\varepsilon)\longrightarrow 0$ as $\varepsilon\longrightarrow 0$.
    \begin{itemize}

    \item[(i)] 
    Let $x,y\in X$ such that $x\perp_{D^\beta}y$. Then
    \begin{align*}
      \|Tx+\lambda Ty\|&=\|T(x+\lambda y)\|\\&\geq |\alpha|(1-\delta_1(\varepsilon))\|x+\lambda y\|\\&\geq|\alpha|(1-\delta_1(\varepsilon))\sqrt{1-\beta^2}\|x\|\\&\geq \dfrac{|\alpha|(1-\delta_1(\varepsilon))}{|\alpha|(1+\delta_2(\varepsilon))}\sqrt{1-\beta^2}\|Tx\|\\&=\left(\dfrac{1-\delta_1(\varepsilon)}{1+\delta_2(\varepsilon)}\right)\sqrt{1-\beta^2}\|Tx\|
    \end{align*}
    for all $\lambda\in\mathbb{R}$.

    Let $\eta\in[0,1)$ such that $\sqrt{1-\eta^2}=\left(\dfrac{1-\delta_1(\varepsilon)}{1+\delta_2(\varepsilon)}\right)\sqrt{1-\beta^2}$. Then $$\eta=\sqrt{1-\left(\dfrac{1-\delta_1(\varepsilon)}{1+\delta_2(\varepsilon)}\right)^2(1-\beta^2)}$$ and $Tx\perp_{D^\eta} Ty$. This proves the result.

    \item[(ii)] Let $x,y\in X$ such that $x\perp_B^\beta y$. Then it follows from \cite[Theorem 3.1]{WO} that $\|x+\lambda y\|\geq (1-\frac{2\beta}{1+\beta})\|x\|$ for all $\lambda\in\mathbb{R}$. Let $\gamma\in[0,1)$ such that $\sqrt{1-\gamma^2}=1-\frac{2\beta}{1+\beta}$. Then $\gamma=\sqrt{1-(\frac{1-\beta}{1+\beta})^2}$. This shows that $x\perp_{D^\gamma} y$ and it follows from the arguments similar to (i) that 
    $$\|Tx+\lambda Ty\|\geq \left(\dfrac{1-\delta_1(\varepsilon)}{1+\delta_2(\varepsilon)}\right)\sqrt{1-\gamma^2}\|Tx\|$$ for all $\lambda\in\mathbb{R}$. Thus, $Tx\perp_{D^\eta} Ty$, where $\eta=\sqrt{1-\left(\dfrac{1-\delta_1(\varepsilon)}{1+\delta_2(\varepsilon)}\right)^2{\left(\dfrac{1-\beta}{1+\beta}\right)^2}}$.

    \item[(iii)] Let $x,y\in X$ such that $x\perp_Ry$. Then $x\perp_B y$ and $y\perp_B x$. Now, by taking $\beta=0$ in (ii) it follows that $Tx\perp_{D^{\eta}}Ty$ and $Ty\perp_{D^\eta} Tx$, where $\eta=\sqrt{1-\left(\dfrac{1-\delta_1(\varepsilon)}{1+\delta_2(\varepsilon)}\right)^2}$.
     \end{itemize}
\end{proof}

\begin{theorem}\label{scalar multiple}
    Let $X,Y$ be normed linear spaces and $\beta\in[0,1)$. Let $0\not=T\in \mathcal{L}(X,Y)$ be a scalar multiple of an $\varepsilon$-isometry. If $0\not=S\in \mathcal{L}(X,Y)$ such that $\|S-T\|\leq \varepsilon\|T\|$ then
    \begin{itemize}
        \item[(i)] $S$ is $D^\beta D^\eta$-P, where $\eta=\sqrt{1-\left(\dfrac{1-\delta_1(\varepsilon)-\varepsilon(1+\delta_2(\varepsilon))}{(1+\delta_2(\varepsilon))(1+\varepsilon)}\right)^2(1-\beta^2)}$.
        \item[(ii)] $S$ is $B^\beta D^\eta$-P for $\eta=\sqrt{1-\left(\dfrac{1-\delta_1(\varepsilon)-\varepsilon(1+\delta_2(\varepsilon))}{(1+\delta_2(\varepsilon))(1+\varepsilon)}\right)^2{\left(\dfrac{1-\beta}{1+\beta}\right)^2}}$.
        \item[(iii)] $S$ is $R D^\eta$-P and $R D^\eta$-R for $\eta=\sqrt{1-\left(\dfrac{1-\delta_1(\varepsilon)-\varepsilon(1+\delta_2(\varepsilon))}{(1+\delta_2(\varepsilon))(1+\varepsilon)}\right)^2}$.
    \end{itemize}
\end{theorem}

\begin{proof}
Let $T=\alpha U$, where $U$ is an $\varepsilon$-isometry and $\alpha\not= 0$. Then 
    \begin{align}\label{eq2}
    |\alpha|(1-\delta_1(\varepsilon))\|x\|\leq \|Tx\|\leq |\alpha|(1+\delta_2(\varepsilon))\|x\|
    \end{align}
    for all $x\in X$, where $\delta_1(\varepsilon)\longrightarrow 0$ and $\delta_2(\varepsilon)\longrightarrow 0$ as $\varepsilon\longrightarrow 0$.

Also,
\begin{align}\label{eq3}
|\|Sx\|-\|Tx\||\leq \|Sx-Tx\|\leq \varepsilon\|T\|\|x\|\leq \varepsilon|\alpha|(1+\delta_2(\varepsilon))\|x\|
\end{align}
for all $x\in X$.
 Thus, by using Eqns.~(\ref{eq2}) and (\ref{eq3}), we get

 $$|\alpha| (1-\delta_1(\varepsilon)-\varepsilon(1+\delta_2(\varepsilon)))\|x\|\leq \|Sx\|\leq |\alpha|(1+\varepsilon)(1+\delta_2(\varepsilon))\|x\|$$
 for all $x\in X$.

    \begin{itemize}

    \item[(i)] 
    Let $x,y\in X$ such that $x\perp_{D^\beta}y$. Then
    \begin{align*}
      \|Sx+\lambda Sy\|&=\|S(x+\lambda y)\|\\&\geq |\alpha| [1-\delta_1(\varepsilon)-\varepsilon(1+\delta_2(\varepsilon))]\|x+\lambda y\|\\&\geq|\alpha| (1-\delta_1(\varepsilon)-\varepsilon(1+\delta_2(\varepsilon)))\sqrt{1-\beta^2}\|x\|\\&\geq \dfrac{|\alpha| (1-\delta_1(\varepsilon)-\varepsilon(1+\delta_2(\varepsilon)))}{|\alpha|(1+\varepsilon)(1+\delta_2(\varepsilon))}\sqrt{1-\beta^2}\|Sx\|\\&=\dfrac{ (1-\delta_1(\varepsilon)-\varepsilon(1+\delta_2(\varepsilon)))}{(1+\varepsilon)(1+\delta_2(\varepsilon))}\sqrt{1-\beta^2}\|Sx\|
    \end{align*}
    for all $\lambda\in\mathbb{R}$.

    Let $\eta\in[0,1)$ such that $\sqrt{1-\eta^2}=\dfrac{ (1-\delta_1(\varepsilon)-\varepsilon(1+\delta_2(\varepsilon)))}{(1+\varepsilon)(1+\delta_2(\varepsilon))}\sqrt{1-\beta^2}$. Then $\eta=\sqrt{1-\left(\dfrac{1-\delta_1(\varepsilon)-\varepsilon(1+\delta_2(\varepsilon))}{(1+\delta_2(\varepsilon))(1+\varepsilon)}\right)^2(1-\beta^2)}$ and $Sx\perp_{D^\eta} Sy$. This proves the result.

    \item[(ii)] Let $x,y\in X$ such that $x\perp_{B^\beta} y$. Then it follows the arguments used in Theorem~\ref{isometry} that $x\perp_{D^\gamma} y$, where $\gamma=\sqrt{1-(\frac{1-\beta}{1+\beta})^2}$. Now, it follows from the arguments similar to (i) that 
    $$\|Sx+\lambda Sy\|\geq \dfrac{ (1-\delta_1(\varepsilon)-\varepsilon(1+\delta_2(\varepsilon)))}{(1+\varepsilon)(1+\delta_2(\varepsilon))}\sqrt{1-\gamma^2}\|Sx\|$$ for all $\lambda\in\mathbb{R}$. If we take $\eta=\sqrt{1-\left(\dfrac{1-\delta_1(\varepsilon)-\varepsilon(1+\delta_2(\varepsilon))}{(1+\delta_2(\varepsilon))(1+\varepsilon)}\right)^2{\left(\dfrac{1-\beta}{1+\beta}\right)^2}}$ then $Sx\perp_{D^\eta} Sy$.

    \item[(iii)] Let $x,y\in X$ such that $x\perp_Ry$. Then $x\perp_B y$ and $y\perp_B x$. Now, by taking $\beta=0$ in (ii) it follows that $Sx\perp_{D^{\eta}}Sy$ and $Sy\perp_{D^\eta} Sx$, where $\eta=\sqrt{1-\left(\dfrac{1-\delta_1(\varepsilon)-\varepsilon(1+\delta_2(\varepsilon))}{(1+\delta_2(\varepsilon))(1+\varepsilon)}\right)^2}$.
     \end{itemize}
\end{proof}

The following result follows immediately from Theorem~\ref{scalar multiple}.
\begin{corollary}
    Let $X,Y$ be normed linear spaces and $\varepsilon,\beta\in[0,1)$. Let $0\not=T\in \mathcal{L}(X,Y)$ be a scalar multiple of an isometry. If $0\not=S\in \mathcal{L}(X,Y)$ such that $\|S-T\|\leq \varepsilon\|T\|$ then
    \begin{itemize}
        \item[(i)] $T$ is $D^\beta D^\eta$-P, where $\eta=\sqrt{1-(\frac{1-\varepsilon}{1+\varepsilon})^2(1-\beta^2)}$.
        \item[(ii)] $T$ is $B^\beta D^\eta$-P for $\eta=\sqrt{1-(\frac{1-\varepsilon}{1+\varepsilon})^2{\left(\dfrac{1-\beta}{1+\beta}\right)^2}}$.
        \item[(iii)] $T$ is $R D^\eta$-P and $R D^\eta$-R for $\eta=\sqrt{1-(\frac{1-\varepsilon}{1+\varepsilon})^2}$.
        \end{itemize}
\end{corollary}

\section{locally $B D^\varepsilon$-P operators}

In this section, we will study operators with a finite-dimensional domain for which locally $B D^\varepsilon$-P at some points implies that the operator is $B D^\eta$-P, where $\varepsilon,\eta\in[0,1)$.

We recall the following result from \cite{DAY}.

   \begin{theorem}\cite[Theorem 4.1]{DAY}\label{basis}
    Let $X$ be an $n$-dimensional Banach space. Then there exists a basis $\{x_1,x_2,\ldots,x_{i-1},x_i,x_{i+1},\ldots,x_n\}$ of $X$ such that $x_i\perp_B \mbox{span}\{x_1,\\x_2,\ldots,x_{i-1},x_{i+1},\ldots,x_n\}$ for all $1\leq i \leq n$.
\end{theorem}

We fix the following notation. Let $X$ an $n$-dimensional Banach space. We say that a basis $\{x_1,x_2,\ldots,x_{i-1},x_i,x_{i+1},\ldots,x_n\}$ of $X$ satisfies the property $(*)$ if $x_i\perp_B \mbox{span}\{x_1,x_2,\ldots,x_{i-1},x_{i+1},\ldots,x_n\}$ for all $1\leq i \leq n$.

\begin{theorem}\label{r4}
    Let $X$ be an $n$-dimensional Banach space and $\{x_1,x_2,\ldots,x_{i-1},x_i,\\x_{i+1},\ldots,x_n\}$ be a basis of $X$ that satisfies property $(*)$. Let $Y$ be a normed linear space. If $T\in S_{\mathcal{L}(X,Y)}$ is $BD^{\varepsilon_i}$-P at each $x_i$ for some $\varepsilon_i\in[0,1)$ and $Tx_i\not=0$ for all $i$ then $T$ is $B D^{\eta}$-P for some $\eta\in[0,1)$.
\end{theorem}
\begin{proof}
    It follows from Theorem~\ref{one} that to prove the result it is sufficient to show that $T$ is injective. If $T$ is not injective, then there exists an $x\in S_X$ such that $Tx=0$. Let $x=\sum_{i=1}^n \alpha_ix_i$, where $\alpha_i\in\mathbb{R}$, $1\leq i\leq n$. Since $Tx_i\not=0$ for each $i$ and $x\in S_X$, we get $\alpha_i$ is non-zero for at least two values of $i$. Let $1\leq i_1\leq n$ such that $\alpha_{i_1}\not=0$. Then $\alpha_{i_1}Tx_{i_1}=-\sum_{i\not=i_1}\alpha_i Tx_i$. Now, by using property $(*)$ of the given basis we get $x_{i_1}\perp_B \sum_{i\not=i_1}\alpha_i x_i$. Thus,  $Tx_{i_1}\perp_{D^{\varepsilon_{i_1}}} \sum_{i\not=i_1}\alpha_i Tx_i$ for some $\varepsilon_{i_1}\in[0,1)$, a contradiction. Thus, $T$ is one-to-one and the result follows. 
\end{proof}

Let $\{e_1,e_2,\ldots,e_n\}$ denotes the standard unit vector basis of $\ell_1^n$. The basis $\{e_1,e_2,\\\ldots,e_n\}$ of $\ell_1^n$ satisfies property $(*)$. In the following result we show that if any operator $T$ from $\ell_1^n$ to any normed linear space $Y$ is $BD^\varepsilon$-P at any one of the basis element $e_i$ with $Te_i\not=0$ then $T$ is non-zero at each of the basis elements $e_j$, where $\varepsilon\in[0,1)$.

\begin{proposition}\label{r5}
    Let $Y$ be a normed linear space. Let $T\in S_{\mathcal{L}(\ell_1^n,Y)}$ and $1\leq i\leq n$ such that $T$ is $BD^\varepsilon$-P at $e_i$ and $Te_i\not=0$. Then $Te_j\not=0$ for all $1\leq j\leq n$.
\end{proposition}

\begin{proof}
 We first observe that $e_i\perp_B e_i+e_j$ for all $j\not=i$. Let $Te_{j_0}=0$ for some $1\leq j_0\leq n$. By using  $T$ is $BD^\varepsilon$-P at $e_i$, $Te_i\not=0$ and $e_i\perp_B e_i+e_{j_0}$, we get $Te_i\perp_{D^\varepsilon} Te_i$, a contradiction. Thus, $Te_j\not=0$ for all $1\leq j\leq n$.
\end{proof}

As a consequence of Theorem~\ref{r4} and Proposition~\ref{r5} we now prove the following results.
\begin{corollary}\label{c2}
    Let $Y$ be a normed linear space. Let $T\in {\mathcal{L}(\ell_1^n,Y)}$ such that $T$ is $BD^{\varepsilon_i}$-P at each $e_i$ for some $\varepsilon_i\in[0,1)$. Then one of the following is true.
    \begin{itemize}
        \item[(i)] $T=0$
        \item[(ii)]$T$ is one-to-one.
    \end{itemize}
\end{corollary}

\begin{proof}
    If $T=0$ then (i) follows.
    Let $T\not=0$, there exists at least one $i$, $1\leq i\leq n$ such that $Te_i\not=0$. Then it follows from Proposition~\ref{r5} that $Te_k\not=0$ for each $1\leq k\leq n$. Since $\{e_1,e_2,\ldots,e_n\}$ satisfies property $(*)$, by using Theorem~\ref{r4}, we get $T$ is one-to-one.
    \end{proof}

\begin{corollary}\label{c3}
    Let $Y$ be a normed linear space. Let $T\in S_{\mathcal{L}(\ell_1^n,Y)}$ such that $T$ is $BD^{\varepsilon_i}$-P at each $e_i$ for some $\varepsilon_i\in[0,1)$. Then $T$ is $BD^\eta$-P for some $\eta\in[0,1)$.
\end{corollary}

We now show that if $T\in \mathcal{L}(X,Y)$ is $BD^\varepsilon$-P at an extreme point $x\in S_X$ and $Tx\not=0$ then $T$ is $BD^\eta$-P, where $\varepsilon,\eta\in[0,1)$, $X$ is a two-dimensional polyhedral Banach and $Y$ is a normed linear space.

\begin{theorem}\label{r6}
    Let $X$ be a two-dimensional polyhedral Banach space and $Y$ be a normed linear space. If $T\in S_{\mathcal{L}(X,Y)}$ is $BD^\varepsilon$-P at an extreme point $x\in S_X$ and $Tx\not=0$ then  $T$ is $BD^\eta$-P for some $\eta\in[0,1)$, where $\varepsilon\in[0,1)$.
\end{theorem}    

\begin{proof}
    It follows from Theorem~\ref{one} that to prove the result, it is sufficient to show that $T$ is injective. If $T$ is not injective, then we get Range $T=span\{Tx\}$. Let $y_1,y_2\in S_X$ be linearly independent elements such that $x\perp_B y_1$ and $x\perp_B y_2$. Then we get $Tx\perp_{D^\varepsilon} Ty_1$ and $Tx\perp_{D^\varepsilon} Ty_2$. This clearly shows that $Ty_1=Ty_2=0$, a contradiction. Thus, $T$ is one-to-one and the result follows.
\end{proof}


\end{document}